\documentclass{amsart}[11pt]

\def\volume{\operatorname{vol}}
\def\vol{\operatorname{vol}}
\def\op{\operatorname}
\def\svolball#1#2{{\volume(\underline B_{#2}^{#1})}}
\def\svolann#1#2{{\volume(\underline A_{#2}^{#1})}}
\def\sball#1#2{{\underline B_{#2}^{#1}}}
\def\svolsp#1#2{{\volume(\partial \underline B_{#2}^{#1})}}

\usepackage[usenames,dvipsnames]{color}
\usepackage{amsmath,amssymb,bm,amsthm, amsfonts, mathrsfs, eucal, epsfig}
\usepackage{graphicx}
\usepackage{url}
\allowdisplaybreaks
\usepackage[top=2.5cm,bottom=2.0cm,left=2.5cm,right=2.5cm]{geometry}

\makeatletter 
\@addtoreset{equation}{section}
\makeatother  

\begin{document}

\newtheorem{Thm}{Theorem}[section]
\newtheorem{Def}{Definition}[section]
\newtheorem{Lem}[Thm]{Lemma}
\newtheorem{Rem}{Remark}[section]

\newtheorem{Cor}[Thm]{Corollary}
\newtheorem{sublemma}{Sub-Lemma}
\newtheorem{Prop}{Proposition}[section]
\newtheorem{Example}{Example}[section]
\newcommand{\g}[0]{\textmd{g}}
\newcommand{\pr}[0]{\partial_r}
\newcommand{\dif}{\mathrm{d}}
\newcommand{\bg}{\bar{\gamma}}
\newcommand{\md}{\rm{md}}
\newcommand{\cn}{\rm{cn}}
\newcommand{\sn}{\rm{sn}}
\newcommand{\seg}{\mathrm{seg}}

\newcommand{\Ric}{\mbox{Ric}}
\newcommand{\Iso}{\mbox{Iso}}
\newcommand{\ra}{\rightarrow}
\newcommand{\Hess}{\mathrm{Hess}}
\newcommand{\RCD}{\mathsf{RCD}}

\title{Improved relative volume comparison for integral Ricci curvature and applications to volume entropy}
\author{Lina Chen}
\address[Lina Chen]{Department of mathematics, Nanjing University, Nanjing China}

\email{chenlina\_mail@163.com}
\thanks{Chen partially supported by the NSFC 12001268 and a research fund from Nanjing University.} 
\author{ Guofang Wei}
\address[Guofang Wei]{Mathematics Department, University of California, Santa Barbara, CA 93106}

\email{wei@math.ucsb.edu}
\thanks{GW partially supported by NSF DMS 1811558 }

\maketitle

\begin{abstract}

\setlength{\parindent}{10pt} \setlength{\parskip}{1.5ex plus 0.5ex
minus 0.2ex} 
We give several  Bishop-Gromov relative volume comparisons with integral Ricci curvature which improve the results in \cite{PW1}. Using one of these volume comparisons, we derive an estimate for the volume  entropy in terms of integral Ricci curvature which substantially improves an earlier estimate in \cite{Au2} and give an application on the algebraic entropy of its fundamental group. We also extend the  almost minimal volume rigidity of \cite{BBCG} to integral Ricci curvature. 
  
  \end{abstract}
\section{Introduction}

Consider a complete Riemannian manifold $M$ with dimension $n$. For each $x\in M$ let $\rho\left( x\right) $ denote the smallest
eigenvalue for the Ricci tensor $\op{Ric} : T_xM\to T_xM$. For a constant $H$, let $\rho_H=\max\{-\rho(x)+(n-1)H, 0\},$ the amount of Ricci curvature lying below $(n-1)H$.  Let
$$k(H, p, R)=\sup_{x\in M}\left(\int_{B_R(x)} \rho_H^p dv\right)^{\frac{1}{p}},$$
$$\bar k(H,p)=\left(\frac{1}{\volume(M)}\int_M \rho_H^p dv\right)^{\frac{1}{p}}= \left(-\kern-1em\int_M \rho_H^p dv\right)^{\frac{1}{p}}, \quad \bar k(H,p, R)= \sup_{x\in M}\left(-\kern-1em\int_{B_R(x)} \rho_H^p dv\right)^{\frac{1}{p}}.$$ 
 Then  $\bar k(H, p)$  measures the average amount of Ricci
 curvature lying below a given bound, in this case,  $(n-1)H$, in the $L^p$ sense. Clearly $\bar k(H, p) = 0$ iff $\Ric_M \ge (n-1)H$. We view $\bar k(H, p)$ as the error term for Ricci curvature bounded below by $(n-1)H$.
 
 For a manifold with Ricci curvature bounded from below, the basic property is the Laplacian and relative volume comparison. Namely if $\op{Ric}_M\geq (n-1)H$, then $\Delta r \le  \underline{\Delta}_H r$, where $r$ is the distance function from a point. Let   $\psi=\max\{\Delta r - \underline{\Delta}_Hr, 0\}$. Then $\op{Ric}_M\geq (n-1)H$ implies $\psi \equiv 0$. In \cite{PW1} Petersen-Wei generalized this result by obtaining an estimate for $\psi$ without any curvature assumption, and thus also gave a generalization of the volume comparison. On the other hand the estimate does not work well for large balls as the error grows when the radius grows.

 In this paper we improve the estimate in \cite[Lemma 2.1]{PW1} which enables us to derive several Bishop-Gromov type relative volume comparisons which are suitable for large balls. In particular the constants in the error term have explicit dependence on the radius $R$ or are  independent of $R$, which is  very useful for big balls. We also express the error term in terms of the average amount of $\rho_H^p$, which is often useful for applications.
 
   Here are our first comparison estimates. Denote  $\svolball{H}{r}$  the volume of the ball with radius $r$ in the $n$-dimensional simply connected Riemannian manifold with constant sectional curvature $H$,  called the model space $\mathbb M^n_H$.
 \begin{Thm} \label{thm-double-R}  Given $n, p> \frac{n}{2}$, $H \le 0$, $R>r>0$,  for any complete Riemannian $n$-manifold $M$ and $x \in M$, we have 
 	\begin{equation}
 	\left(\frac{\volume(B_R(x))}{\svolball{H}{R}}\right)^{\frac1{2p}} - 
\left(\frac{\volume(B_r(x))}{\svolball{H}{r}}\right)^{\frac1{2p}}\le  C(n, p, H, R)\, k(H, p, R)^{\frac12}, \label{eq-double-R}
 	\end{equation}
 	\begin{equation}\frac{\volume(\partial B_R(x))}{\volume(B_R(x))}\leq \frac{\svolsp{H}{R}}{\svolball{H}{R}} + c(n, p) \, \bar k(H, p, R)^{\frac12},   \label{sph-com} \end{equation}
	where $c(n, p)=\left(\frac{(n-1)(2p-1)}{2p-n}\right)^{1/2}$, \ $C(n,p,H,R)= \frac{c(n,p)}{2p} \int_{0}^{R}\left(\frac{1}{\svolball{H}{t}}\right)^{\frac1{2p}}dt$. 
 \end{Thm}
\begin{Rem}
	Note that the constant $C(n,p,H,R)$ in (\ref{eq-double-R})
	is bounded for all $R$ when $H<0$, which is a big improvement of Theorem 1.1 in \cite{PW1}, where the same estimate is obtained, but the constant $C(n, p, H, R)$ goes to infinite as $R \to \infty$ for $H<0$. 
\end{Rem}

 In Theorem~\ref{thm-double-R} the integral of $\rho^p_H$, Ricci curvature error term,  is on the bigger ball, and the comparison estimate is on the smaller or same size ball. Below we give an estimate where the error is on  the smaller ball and the estimate is for all big balls.

 \begin{Thm}\label{main0}
Given $n, p> \frac{n}{2}$, $H \le 0$, $R>r\geq1$, there exists $c(n, p, H)>0, \delta_0=\delta(n, H, p)>0$, such that if a complete Riemannian $n$-manifold $M$ satisfies $\bar k(H, p,1)\leq \delta\leq \delta_0$, then the following holds:
\begin{equation}\frac{\volume(B_R(x))}{\svolball{H}{R}}\leq e^{(R-r)c(n, p, H)\delta^{\frac{1}{2}}}\frac{\volume(B_r(x))}{\svolball{H}{r}}.
\label{rvol-com}
\end{equation}
Moreover,
\begin{equation}\volume(B_R(x))\leq \left(1+c(n, p, H)\delta^{\frac12}\right)e^{(R-1)c(n, p, H)\delta^{\frac12}}\svolball{H}{R}.\label{vol-com}\end{equation}
\end{Thm} 
\begin{Rem}
{\rm (a)} Under the assumption $\bar k(H, p, 1)\leq \delta_0$, the comparison inequalities \eqref{rvol-com} and \eqref{vol-com} hold for $R>r\geq1$. For $1\geq R>r>0$, the relative volume comparison also holds, see the doubling property \eqref{vol-double}. 

{\rm (b)} Theorem~\ref{main0} holds for any fixed $r_0$ instead of $1$, with  $ \delta_0$ and $c$ also depend on $r_0$ and $R>r\geq r_0$, see \eqref{grvol-com} and \eqref{gvol-com} for detail. 
\end{Rem}

 As an immediate application of \eqref{rvol-com}, we get comparison of any two balls when the difference between their radius is bounded such that the error term is independent of the radius.  Namely
\begin{Cor}
Let the assumption be as in Theorem~\ref{main0}, then for each fixed $D>1$ and arbitrary large $L>2D$, we have that 
$$\frac{\volume(B_{L}(x))}{\svolball{H}{L}}\leq e^{c(n, p, H)D\delta^{\frac12}}\frac{\volume(B_{L-D}(x))}{\svolball{H}{L-D}}.$$
\end{Cor} 

See Corollary~\ref{sph-comp} for another version of volume comparison.

Combining  the estimate (\ref{sph-com}) or (\ref{vol-com}) with Aubry's estimate on $ \bar k(H, p, R)$ for the universal cover (see (\ref{univ-rho})), we immediately have the following optimal estimate of the volume entropy for integral Ricci curvature:
 \begin{Thm} \label{main}
	Given $n, \ p>\frac{n}{2}, d>0, H\leq 0$, there exists  $\delta_0=\delta(n, p, d, H)>0,\ c(n,p,H,d) >0$, such that if a compact Riemannian $n$-manifold $M$ satisfies 
	$\bar k(H, p)\leq \delta\leq \delta_0$, $\op{diam}(M)\leq d$ ,
	then the volume entropy of $M$ satisfies that
	\begin{equation}
	h(M)\leq (n-1)\sqrt{-H}+ c(n, p, H,d)\, \delta^{\frac{1}{2}}.  \label{h-est-int}
		\end{equation}
where \begin{equation}
h(M) = \lim_{ R \ra \infty} \frac{\ln \vol (B_R(\tilde{x}))}{R},
\end{equation}
and $B_R(\tilde{x})$ is a ball in $\tilde{M}$, the universal cover of $M$.	
\end{Thm}

Volume entropy is  a fundamental geometric invariant. It is related to the topological entropy, minimal volume, simplicial volume, bottom spectrum of the Laplacian of the universal cover, among others.   For a compact Riemannian $n$-manifold $M$, the limit in (1.6) exists and is independent of the base point $\tilde{x} \in \tilde{M}$ \cite{Ma}. 
\begin{Rem}
{\rm (a)} For $M$ with $\op{Ric}_M\geq (n-1)H$,  Bishop volume comparison gives the upper bound 
\begin{equation*}
h(M)\begin{cases}= 0, & H\geq 0;\\
\leq (n-1)\sqrt{-H}, & H< 0. \end{cases}  \label{h-est}
\end{equation*}

{\rm (b)} For $M$ with integral Ricci curvature bound, $\bar k(H, p)\leq \delta$, 

{\rm (b1)} when $H>0$, Aubry \cite{Au1} showed that the diameter of $M$ is bounded and its fundamental group is finite, hence $h(M)=0$;

{\rm (b2)} when $H\leq 0$, in \cite{Au2}, Aubry derived that
$$h(M)\leq c(n, p)e^{\frac{(n-1)\sqrt{-H} d}{p}}\left(\delta-(n-1)H\right)^{\frac{1}{2}},$$
 where $c(n,p)=2^{\frac{1}{2p}}(n-1)^{\frac{p-1}{2p}}\left(\frac{4p(p-1)}{(2p-1)(2p-n)}\right)^{\frac{p-1}{2p}}3^{\frac{n+1}{2p}}$. 
 Thus for $H\leq 0$,  one only obtains $h(M)\leq (n-1)\sqrt{-H}$ when $\delta\to 0$ and  $p\to \infty$.  From \eqref{h-est-int} we have $h(M)\leq (n-1)\sqrt{-H}$ when $\delta\to 0$  for each $p > \frac n2$.   Therefore our estimate improves the estimate in \cite{Au2}. In fact it is optimal. 
 
{\rm (c)} Paeng \cite{Pa} had a rough estimate of volume entropy under strong extra assumption.

\end{Rem}

As an application of Theorem~\ref{main} we generlize Bessi\`eres, Besson, Courtois and Gallot's  quantitative minimal volume entropy rigidity \cite{BBCG} to integral Ricci curvature case. For a compact Riemannian $n$-manifold $M$ with $\op{Ric}_M\geq -(n-1)$, if there is a degree one continuous map from $M$ to a hyperbolic manifold $Y$, then $\volume(M)\ge \volume(Y)$ and equality holds iff $M$ is isometric to $Y$ (\cite{BCG}, see also Theorem~\ref{min-vol-entr}). A quantive version is given in \cite{BBCG}. 
Using Theorem~\ref{main}, and following the argument as in \cite{BBCG} we extend their result to integral curvature bound.  
\begin{Thm}\label{min-vol}
	Given $n\geq 3, p>\frac{n}2,  d>0$, there exists $\epsilon(n, p, d)>0,\  \delta(n, p, d)>0$, such that for a compact hyperbolic $n$-manifold $Y$,  and for $\epsilon< \epsilon(n, p, d),\ \delta<\delta(n, p, d)$, if a compact Riemannian $n$-manifold $M$ satisfies that there is a degree one continuous map $f: M\to Y$ and 
	$$\bar k(-1, p)\leq \delta,  \quad \op{diam}(M)\leq d, \quad \volume(M)\leq (1+\epsilon)\volume(Y),$$
	then $M$ is diffeomorphic to $Y$ by a $\Psi(\epsilon, \delta | n, p, d)$-isometry, where $\Psi(\epsilon, \delta | n, p, d)$ is a constant depends on $\epsilon, \delta, n, p, d$ such that $\Psi(\epsilon, \delta | n, p, d)\to 0$ as $\epsilon, \delta\to 0$ and $n ,p, d$ fixed.
\end{Thm}

Note that in \cite{BBCG},  the diameter upper bound condition is on $Y$ instead of $M$.  

Theorem~\ref{main} also has an application to the algebraic entropy of the fundamental group of $M$. Recall for a finitely generated group $G$ with generating set $S$ satisfying $\gamma\in S$ if and only if $\gamma^{-1}\in S$, the algebraic entropy of $G$ with respect to the generating set $S$ is given by 
$$h(G, S)=\lim_{R\to\infty}\frac{\ln |G(R, S)|}{R}.$$
where  $|G(R, S)|=\#\{\gamma\in G, \, |\gamma|_S\leq R\}$ and $|\gamma|_S=\inf\{t, \, \gamma_{i_1}\cdots\gamma_{i_t}=\gamma, \gamma_{i_j}\in S, j=1,\cdots, t\}$.
The algebraic entropy of $G$ is $$\quad h(G)=\inf_Sh(G, S).$$ 

For a compact Riemannian $n$-manifold $M$ one can choose a generating set $S_0$ of $\pi_1(M)$ such that $ d(\gamma\tilde x, \tilde x)\leq 2\op{diam}(M)+1$ for all $\gamma\in S_0$. With this generating set, algebraic entropy of $\pi_1(M)$ and the volume entropy of $M$ have the following relation  
\begin{equation}
h(M)\leq h(\pi_1(M), S_0)\leq  (2\op{diam}(M)+1)h(M). \label{two entropies}
\end{equation}
This is essentially in \cite{Sv}, cf. \cite{Mi}, see \cite[Lemma 1.2]{CRXa} for a proof.

We call a compact $n$-manifold $M$ almost non-negative Ricci curvature in the integral sense if there is a sequence of metric $g_i$ on $M$, such that $\bar k_{g_i}(0, p)\op{diam}^2(M, g_i)\to 0$ as $i \rightarrow \infty$. Using Theorem~\ref{main} and \eqref{two entropies}, we have
\begin{Cor}  \label{h-pi}
If a compact $n$-manifold $M$ has almost non-negative Ricci curvature in integral sense, then the algebraic entropy of $\pi_1(M)$, $h(\pi_1(M))=0$.
\end{Cor}
\begin{Rem}
{\rm (i)} In fact, by \cite[Corollary 11.17]{BGT} and (\ref{vol-double}), we know that the fundamental groups of manifolds with almost non-negative Ricci curvature in the integral sense are virtually nilpotent.

{\rm (ii)} By Gromov's relation between simplicial volume and minimal volume entropy \cite{Gro}, Theorem~\ref{main} also gives an upper bound on the simplicial volume which improves the estimate in \cite{Au2}. 
\end{Rem}

The paper is organized as follows. In Section 2, we first give the key improvement on the volume ratio estimate in Lemma~\ref{dvol-c}, then give a proof of  Theorem~\ref{thm-double-R}, Theorem~\ref{main0} using the Laplacian comparison estimate in \cite{PW1}. The proofs of Theorem~\ref{main} and Theorem~\ref{min-vol} are in Section 3.


\section{Relative volume comparison estimates for integral Ricci curvature}

To prove Theorems~\ref{thm-double-R} and \ref{main0}, we use the method as in \cite{PW1} and 
the Laplacian comparison proved in there. The key part is to  give an improvement on the estimate of the differential of the volume ratio $\frac{\volume(B_t(x))}{\svolball{H}{t}}$ (see Lemma~\ref{dvol-c}). Then  Theorems~\ref{thm-double-R}  is proved using Laplacian comparison \eqref{lap-com1}. Next the volume doubling  property  is proved, which finishes the proof of Theorem~\ref{main0}.

For a complete Riemannian $n$-manifold $M$, fix $x \in M$, write the volume element $d \vol = \mathcal A (t,\theta)dt \,d\theta$ in term of polar coordinate at $x$. We define $\mathcal A$ to be zero when $t$ increases and $\mathcal A$ becomes undefined.  Let $r =d (x, \cdot)$ be the distance function from $x$ and $\psi =\max\{\Delta r - \underline{\Delta}_Hr, 0\}$, the error in Laplacian comparison. Then

 \begin{Lem}\label{dvol-c}
\begin{equation}\frac{d}{dt}\left(\ln\frac{\volume(B_t(x))}{\svolball{H}{t}}\right)\leq \left(-\kern-1em\int_{B_t(x)} \psi^{2p}\right)^{\frac{1}{2p}}, \label{dvol}\end{equation}
\begin{equation}
\frac{\volume(\partial B_R(x))}{\volume(B_R(x))}\leq \frac{\svolsp{H}{R}}{\svolball{H}{R}}  +  \left(-\kern-1em\int_{B_R(x)} \psi^{2p}\right)^{\frac{1}{2p}}.  \label{area-vol-double}
\end{equation}
\end{Lem}
\begin{Rem}
The key improvement here is that the constant in front of the error term  is $1$. In \cite[Lemma 2.1]{PW1}  the estimate (\ref{dvol})  was given but with constant $c(n, H, t)=\max_{s\in[0, t]}\frac{t\underline{\mathcal{A}}_H(s)}{\int_0^s \underline{\mathcal{A}}_H(\tau)d\tau}$, which goes to infinite as $t \to \infty$ when $H<0$ and $c(n,0,t) = n$.
\end{Rem}

\begin{proof}
Since 	$\frac{d}{dt} {\mathcal A}(t, \theta) = (\Delta r)  \mathcal A(t, \theta)$ and 	$\frac{d}{dt} \underline{\mathcal A}_H(t) = (\underline{\Delta}_H r) \underline{\mathcal A}_H(t)$, 
\[
\frac{d}{dt} \left( \frac{{\mathcal A}(t, \theta)}{\underline{{\mathcal A}}_H(t)}\right)
= (\Delta r-\underline{\Delta}_H r)  \cdot \frac{{\mathcal A}(t, \theta)}{\underline{{\mathcal A}}_H(t)} \le \psi
\cdot \frac{{\mathcal A}(t, \theta)}{\underline{{\mathcal A}}_H(t)}.
\]
Hence 
\begin{eqnarray}
\frac{d}{dt} \left( \frac{\volume(\partial B_t(x))}{\svolsp{H}{t}} \right)  & = & \frac{1}{ \vol S^{n-1}} \int_{S^{n-1}}  \frac{d}{dt} \left( \frac{{\mathcal A}(t, \theta)}{\underline{{\mathcal A}}_H(t)}\right)  d\theta \nonumber\\
& \le &  \frac{1}{\svolsp{H}{t}} \int_{S^{n-1}}  \psi \, {\mathcal
	A}(t,\theta)d\theta.  \label{area-derivative}
\end{eqnarray} 
Integrate \eqref{area-derivative}  with respect to $t$ from $l$ to $r$ and cross multiply,
\begin{equation} \svolsp{H}{l}\volume(\partial B_r(x)) - \svolsp{H}{r}\volume(\partial B_l(x))\leq \svolsp{H}{r}\svolsp{H}{l}\int_l^r\frac{1}{\svolsp{H}{s}}\int_{S^{n-1}}\psi\mathcal{A}d \theta ds.  
\label{area-com} 
\end{equation}
Now integrate with respect to $l$ from $0$ to $r$ gives
\begin{equation}
 \svolball{H}{r}\volume(\partial B_r(x)) - \svolsp{H}{r}\volume(B_r(x))\leq \int_0^r \svolsp{H}{r}\svolsp{H}{l}\int_l^r\frac{1}{\svolsp{H}{s}}\int_{S^{n-1}}\psi\mathcal{A}d \theta\,  ds \, dl.  \label{area-vol-com}
\end{equation}
\begin{eqnarray*}
\frac{d}{dt}\frac{\volume(B_t(x))}{\svolball{H}{t}} &=& \frac{d}{dt}\frac{\int_0^t\volume(\partial B_{\tau}(x))d\tau}{\int_0^{t}\svolsp{H}{\tau}d \tau}\nonumber\\
&=& \frac{\volume(\partial B_t(x))\int_0^t\svolsp{H}{\tau}d\tau -\int_0^t\volume(\partial B_{\tau}(x))d\tau \,  \svolsp{H}{t}}{\left(\svolball{H}{t}\right)^2}\nonumber \\
&\overset{\eqref{area-com}}\leq & \frac{\int_0^t\svolsp{H}{t}\svolsp{H}{\tau}\int_{\tau}^t\frac{1}{\svolsp{H}{s}}\int_{S^{n-1}}\psi \mathcal A(s, \theta)d\theta ds d\tau}{\left(\svolball{H}{t}\right)^2}\nonumber \\
&= & \frac{\int_0^t\int_0^s\int_{S^{n-1}}\svolsp{H}{t}\frac{\svolsp{H}{\tau}}{\svolsp{H}{s}}\psi(s, \theta) \mathcal A(s, \theta)d\theta d\tau ds}{\left(\svolball{H}{t}\right)^2}\nonumber \\
&= & \frac{\svolsp{H}{t}\int_0^t\int_{S^{n-1}}\frac{\svolball{H}{s}}{\svolsp{H}{s}}\psi(s, \theta) \mathcal A(s, \theta)d\theta ds}{\left(\svolball{H}{t}\right)^2}\nonumber \\
&\leq & \frac{\svolsp{H}{t} \left(\max_{s \in [0,t]}\frac{\svolball{H}{s}}{\svolsp{H}{s}} \right) \int_0^t\int_{S^{n-1}}\psi(s, \theta) \mathcal A(s, \theta)d\theta ds}{\left(\svolball{H}{t}\right)^2}\nonumber \\
&\leq & \frac{\svolsp{H}{t} \left(\max_{s \in [0,t]}\frac{\svolball{H}{s}}{\svolsp{H}{s}}\right) \left(\int_0^t\int_{S^{n-1}}\psi^{2p}(s, \theta) \mathcal A(s, \theta)d\theta ds\right)^{\frac1{2p}}\left(\int_0^t\int_{S^{n-1}}\mathcal A(s, \theta)d\theta d s\right)^{1-\frac1{2p}}}{\left(\svolball{H}{t}\right)^2}\nonumber \\
& = & \frac{\svolsp{H}{t}}{\svolball{H}{t}}\left(\max_{s \in[0,t]}\frac{\svolball{H}{s}}{\svolsp{H}{s}}\right) \frac{\volume(B_t(x))}{\svolball{H}{t}}\left(-\kern-1em\int_{B_t(x)} \psi^{2p}\right)^{\frac{1}{2p}}. 
\end{eqnarray*}
Here  we used H\"older's inequality.
Denote
$$f_H(t)=\left(\max_{s \in[0,t]}\frac{\svolball{H}{s}}{\svolsp{H}{s}} \right) \frac{\svolsp{H}{t}}{\svolball{H}{t}}.$$
We have 
\begin{equation*}
\frac{d}{dt}\left(\ln\frac{\volume(B_t(x))}{\svolball{H}{t}}\right)\leq f_H(t) \left(-\kern-1em\int_{B_t(x)} \psi^{2p}\right)^{\frac{1}{2p}}. \label{dvol2}
\end{equation*}

Claim: $f_H(t) \equiv 1$ for all $H \in \mathbb R$ ($t< \frac{\pi}{\sqrt H}$ when $H>0$). 

 For $H=0$, this is obvious. For $H<0$,
let $$g_{-1}(s)=\frac{\int_0^s\sinh^{n-1}\tau d\tau}{\sinh^{n-1}s}.$$
 Then 
 \begin{eqnarray*}
 g'_{-1}(s)&=&\frac{\sinh^{2(n-1)}s-(n-1)\sinh^{n-2}s \cosh s\int_0^s\sinh^{n-1}\tau}{\sinh^{2(n-1)}s}\\
 &=& \frac{\cosh s}{\sinh^{n}s}\left(\frac{\sinh^n s}{\cosh s}-(n-1)\int_0^s \sinh^{n-1}\tau d\tau\right).
\end{eqnarray*}
 Let 
 $$h_{-1}(s)=\frac{\sinh^n s}{\cosh s}-(n-1)\int_0^s \sinh^{n-1}\tau d\tau.$$
 Then
 \begin{eqnarray*}
 h'_{-1}(s)& = & n\sinh^{n-1}s-\frac{\sinh^{n+1}s}{\cosh^2s}-(n-1)\sinh^{n-1}s\\
 &= & \frac{\sinh^{n-1}s}{\cosh^2s}\left(\cosh^2s-\sinh^2s\right)\\
 & =& \frac{\sinh^{n-1}s}{\cosh^2s}\geq0.
  \end{eqnarray*}
 Since $h_{-1}(0)=0$ we have $h_{-1}(s)\geq 0$ and thus $g'_{-1}(s)\geq 0.$
 Hence
 $$f_{-1}(t)=\max_{s\in[0, t]}\frac{\int_0^s\sinh^{n-1}\tau d\tau}{\sinh^{n-1}s}\, \frac{\sinh^{n-1}t}{\int_0^t\sinh^{n-1}t}=1.$$
 Since $f_{H}(t)=f_{-1}(\sqrt{-H}t)$, we have $f_H(t)=1$.
 
 For $H=1, s<\pi/2$, similarly, we have that $h'_{1}(s)=\frac{\sin^{n-1}s}{\cos^2(s)}>0$ and thus $g'_1(s)\geq 0$. Then $f_1=1$. And for $s\in (\pi/2, \pi)$, it follows from the definition directly. Then $f_H(t)=1$ for $H>0$ and $t < \frac{\pi}{\sqrt H}$.
 
Then \eqref{dvol} follows.

To prove (\ref{area-vol-double}), we change $r$ to $R$ in  (\ref{area-vol-com}) and divide by $\svolball{H}{R} \, \volume(B_R(x))$, and do change of integral on the right hand side as above gives
\begin{eqnarray*}
\frac{\volume(\partial B_R(x))}{\volume(B_R(x))} -  \frac{\svolsp{H}{R}}{\svolball{H}{R}}  &	\leq & \frac{\svolsp{H}{R}}{\svolball{H}{R}} \frac{1}{\volume(B_R(x))} \int_0^R \int_0^s \int_{S^{n-1}}\frac{\svolsp{H}{l}}{\svolsp{H}{s}}\psi(s, \theta) \mathcal A(s, \theta)d\theta dl\, ds \\
& \le & \frac{\svolsp{H}{R}}{\svolball{H}{R}}\left(\max_{s \in[0,R]}\frac{\svolball{H}{s}}{\svolsp{H}{s}} \right)  \left(-\kern-1em\int_{B_R(x)} \psi\right) \\
& \leq & f_H(R) \left(-\kern-1em\int_{B_R(x)} \psi^{2p}\right)^{\frac{1}{2p}}.
\end{eqnarray*}
This gives (\ref{area-vol-double}) since $f_H(R)=1$. 
\end{proof}

Now the right hand side of (\ref{dvol}) is controlled by the Laplacian comparison estimate  in \cite[Lemma 2.2]{PW1}.

\begin{Thm} \label{com-int}
	Given $n, p> \frac{n}{2}$, $H\le 0$, $r >0$,  
	\begin{eqnarray*}
	\int_{0}^{r}\left(\psi\right)^{2p}\mathcal A(t, \theta)dt & \leq &  \left(\frac{(n-1)(2p-1)}{2p-n}\right)^p \int_{0}^{r} \rho_H^p\, \mathcal A(t, \theta)dt. 
	\end{eqnarray*}
	This gives
	\begin{equation}
	\left(-\kern-1em\int_{B_r(x)} \psi^{2p}\right)^{\frac{1}{2p}} \le c(n,p) \bar k(H,p, r)^{1/2},
 \label{lap-com1}	\end{equation}
 where $c(n,p) = \left(\frac{(n-1)(2p-1)}{2p-n}\right)^{1/2}$. 
\end{Thm}

Combining the improved estimates in Lemma~\ref{dvol-c} with (\ref{lap-com1}) we can prove Theorem~\ref{thm-double-R}. 
\begin{proof}[Proof of Theorem~\ref{thm-double-R}]
	Rewrite \eqref{dvol} as
	$$\frac{d}{dt}\frac{\volume(B_t(x))}{\svolball{H}{t}}\leq \frac{1}{\left(\svolball{H}{t}\right)^{\frac1{2p}}}\left(\frac{\volume(B_t(x))}{\svolball{H}{t}}\right)^{1-\frac1{2p}}\left(\int_{B_t(x)}\psi^{2p}\right)^{\frac1{2p}}.$$
	Integrate both sides from $r_1$ to $r_2$ gives,
	\begin{equation}\left(\frac{\volume(B_{r_2}(x))}{\svolball{H}{r_2}}\right)^{\frac1{2p}}-\left(\frac{\volume(B_{r_1}(x))}{\svolball{H}{r_1}}\right)^{\frac1{2p}}\leq \frac{1}{2p}  \left(\int_{r_1}^{r_2}\frac{1}{\left(\svolball{H}{t}\right)^{\frac1{2p}}}dt
	\right) \left(\int_{B_{r_2}(x)}\psi^{2p}\right)^{\frac1{2p}}. \label{relative-2p} \end{equation}
This gives for $0<r \le R$, 
\begin{equation}
\left(\frac{\volume(B_R(x))}{\svolball{H}{R}}\right)^{\frac1{2p}} - 
\left(\frac{\volume(B_r(x))}{\svolball{H}{r}}\right)^{\frac1{2p}}\le  c(n, p, H, R)\, \left(\int_{B_{R}(x)}\psi^{2p}\right)^{\frac1{2p}}, \label{pw-improve}
\end{equation}
where 
$$c(n,p,H,R) = \frac{1}{2p} \int_{0}^{R}\left(\frac{1}{\svolball{H}{t}}\right)^{\frac1{2p}}dt.$$ Since $p> \frac n2$ the integral converges at $0$. Moreover 
	\[ c(n,p,H,R)  \left\{ \begin{array}{lll} = & 
 \frac{R^{1-\frac{n}{2p}}}{w^{1/{2p}}_n(2p-n)} & \mbox{when}  \ H=0 \\
 \le &  c(n,p, H) & \mbox{when} \ H< 0  \end{array} \right.	 \] since  when $H<0$, $\lim_{R\to \infty}c(n, p, H, R)$ is finite so $c(n,p,H,R)$ is bounded for all $R$.

	Now apply (\ref{lap-com1}) to (\ref{pw-improve}) gives
	\begin{equation}  \label{eq-relative-k} 
\left(\frac{\volume(B_R(x))}{\svolball{H}{R}}\right)^{\frac1{2p}} - 
\left(\frac{\volume(B_r(x))}{\svolball{H}{r}}\right)^{\frac1{2p}}\le  C(n, p, H, R)\, k(H, p, R)^{\frac12},	\end{equation} with
	\begin{equation}  \label{C}
	C(n, p, H, R) =c(n,p) \cdot c(n, p, H, R).
	\end{equation}	 
This is \eqref{eq-double-R}.
	  
Applying the Laplacian comparison \eqref{lap-com1} to \eqref{area-vol-double}, \eqref{sph-com} follows directly.
\end{proof}	

In order to prove Theorem~\ref{main0} we need a volume doubling estimate. 
In \cite[Theorem 2.1]{PW2}, Petersen-Wei proved the doubling property in manifolds with bounded integral Ricci curvature for $H=0$. Here using (\ref{relative-2p}), we present a proof  of 
doubling property for general $H$ following the argument in \cite[Theorem 2.1]{PW2}.  
Our doubling constant is explicit and much improved. 

\begin{Lem}  \label{doubling}
Given $n, p>\frac{n}{2}, H \le 0$ and $R>0$, if a complete Riemannian $n$-manifold $M$ satisfies $\bar k(H,p, R)\leq \delta$ with $C_1(n, p, H, R)\delta^{\frac12} \le \frac 12 \left(1 - \frac{1}{2^{1/2p+1}-1}\right)$, where $C_1(n,p,H,R)$ is given in (\ref{C1}), then for any $0<r_1< r_2\leq R$,
\begin{equation}
\frac{\volume(B_{r_1}(x))}{\volume(B_{r_2}(x))} \ge \left( \frac{1-2C_1(n, p, H, R)\delta^{\frac12}}{1-C_1(n, p, H, R)\delta^{\frac12}} \right)^{2p} \frac{\svolball{H}{r_1}}{\svolball{H}{r_2}}\ge 
 \frac12\frac{\svolball{H}{r_1}}{\svolball{H}{r_2}}.  \label{vol-double}
\end{equation}
\end{Lem}
\begin{proof}
Multiply both sides of (\ref{relative-2p}) by $\left(\frac{\svolball{H}{r_1}}{\volume(B_{r_2}(x))}\right)^{\frac1{2p}}$, we have  
\begin{equation}\left(\frac{\svolball{H}{r_1}}{\svolball{H}{r_2}}\right)^{\frac1{2p}}-\left(\frac{\volume(B_{r_1}(x))}{\volume(B_{r_2}(x))}\right)^{\frac1{2p}}\leq \frac{1}{2p} \left(\frac{\svolball{H}{r_1}}{\volume(B_{r_2}(x))}\right)^{\frac1{2p}}\left(\int_{r_1}^{r_2}\frac{1}{\left(\svolball{H}{t}\right)^{\frac1{2p}}}dt \right) \left(\int_{B_{r_2}(x)}\psi^{2p}\right)^{\frac1{2p}}. \label{double-error} \end{equation}
Let
$$c(r)= \frac{1}{2p} \left(\frac{\svolball{H}{r}}{\volume(B_{r}(x))}\right)^{\frac1{2p}}\left(\int_{r_1}^{r}\frac{1}{\left(\svolball{H}{t}\right)^{\frac1{2p}}}dt\right) \left(\int_{B_r(x)}\psi^{2p}\right)^{\frac1{2p}}.$$
Then (\ref{double-error}) becomes
$$(1-c(r_2))\left(\frac{\svolball{H}{r_1}}{\svolball{H}{r_2}}\right)^{\frac1{2p}}\leq\left(\frac{\volume(B_{r_1}(x))}{\volume(B_{r_2}(x))}\right)^{\frac1{2p}}.$$
Now we estimate $c(r_2)$. Given any $0<r \le R$, let $r_1 =r,\ r_2=R$ in  \eqref{double-error} or \eqref{relative-2p} and cross product gives, 
\begin{eqnarray*} 
	& &\left(\frac{\svolball{H}{r}}{\volume(B_{r}(x))}\right)^{\frac1{2p}} \\
	& \leq & \left(\frac{\svolball{H}{R}}{\volume(B_{R}(x))}\right)^{\frac1{2p}}\left(1-\frac{1}{2p} \left(\frac{\svolball{H}{R}}{\volume(B_R(x))}\right)^{\frac1{2p}}\int_{r}^{R}\frac{1}{\left(\svolball{H}{t}\right)^{\frac1{2p}}}dt\left(\int_{B_{R}(x)}\psi^{2p}\right)^{\frac1{2p}}\right)^{-1}\nonumber\\
	&\leq &\left(\frac{\svolball{H}{R}}{\volume(B_{R}(x))}\right)^{\frac1{2p}}\left(1-\frac{1}{2p} \left(\int_{0}^{R}\left(\frac{\svolball{H}{R}}{\svolball{H}{t}}\right)^{\frac1{2p}}dt \right) \left(-\kern-1em\int_{B_{R}(x)}\psi^{2p}\right)^{\frac1{2p}}\right)^{-1}\nonumber\\
	&\le & \left(\frac{\svolball{H}{R}}{\volume(B_{R}(x))}\right)^{\frac1{2p}}\left(1-C_1(n, p, H, R)\,  \bar k(H, p, R)^{\frac1{2}}\right)^{-1}, \nonumber
	\end{eqnarray*}
	where \begin{equation} \label{C1}
	C_1(n,p,H,R) = \frac{c(n,p)}{2p} \int_{0}^{R}\left(\frac{\svolball{H}{R}}{\svolball{H}{t}}\right)^{\frac1{2p}}dt.
	\end{equation}
If $\bar k(H, p, R)\leq \delta$, and $C_1(n, p, H, R)\delta^{\frac12} <1$ we have
\begin{equation}\left(\frac{\svolball{H}{r}}{\volume(B_{r}(x))}\right)^{\frac1{2p}}\leq \frac{1}{1-C_1(n, p, H, R)\delta^{\frac12}}\left(\frac{\svolball{H}{R}}{\volume(B_{R}(x))}\right)^{\frac1{2p}}. \label{eq-relative} \end{equation}
Hence
\begin{eqnarray*}c(r_2)&\leq &\frac{1}{1-C_1(n, p, H, R)\delta^{\frac12}}\left(\frac{\svolball{H}{R}}{\volume(B_{R}(x))}\right)^{\frac1{2p}} \left( \frac{1}{2p} \int_{0}^{R}\frac{1}{\left(\svolball{H}{t}\right)^{\frac1{2p}}}dt\right) \left(\int_{B_R(x)}\psi^{2p}\right)^{\frac1{2p}}\\
&\leq &\frac{1}{1-C_1(n, p, H, R)\delta^{\frac12}} \frac{1}{2p} \int_{0}^{R}\left(\frac{\svolball{H}{R}}{\svolball{H}{t}}\right)^{\frac1{2p}}dt\left(-\kern-1em\int_{B_{R}(x)}\psi^{2p}\right)^{\frac1{2p}} \\
&\leq & \frac{C_1(n, p, H, R)\delta^{\frac12}}{1-C_1(n, p, H, R)\delta^{\frac12}}.
\end{eqnarray*}
Here we used Laplacian comparison (\ref{lap-com1}) in the last inequality. 
Now take $\delta$ such that $C_1(n, p, H, R)\delta^{\frac12}\leq \frac 12 \left(1 - \frac{1}{2^{1/2p+1}-1}\right)$ gives $\left( \frac{1-2C_1(n, p, H, R)\delta^{\frac12}}{1-C_1(n, p, H, R)\delta^{\frac12}} \right)^{2p} \ge \frac 12$, which completes the proof. 
\end{proof}

By the volume doubling property Lemma~\ref{doubling} and a simple packing argument, we can compare $\bar k(H, p, r)$ and $\bar k(H, p, R)$ for any $r\leq R$. In fact, in \cite[Page 461]{PW2}, they proved that
\begin{Lem}[\cite{PW2}] \label{pack}
Given $n, p>\frac{n}{2}, H\leq 0$ and $r>0$. Let $\delta=\delta(n, p, H, r)>0$ be as in Lemma~\ref{doubling}. Then if a complete Riemannian $n$-manifold $M$ satisfies $\bar k(H, p, r)\leq \delta$, then for any $R\geq r$, we have that
\begin{equation}\bar k(H, p, R)\leq B(n, H, r)^2 \, \bar k(H, p, r),  \label{c-com}\end{equation}
where $B(n, H, r) = \left(2\frac{\svolball{H}{r}}{\svolball{H}{\frac{r}{2}}}\right)^{\frac{1}{2p}}$. 
\end{Lem} 
Now using \eqref{dvol}, Lemma~\ref{doubling}, Laplacian comparison \eqref{lap-com1} and \eqref{c-com}, we are ready to prove Theorem~\ref{main0}. 

\begin{proof}[Proof of Theorem~\ref{main0}] Let $\delta_0=\delta(n, p, H, r_0)$ be as in Lemma~\ref{doubling}. Let $\bar k(H, p, r_0)\leq \delta_0$. Then by Lemma~\ref{pack}, we have that 
for each $t\geq r_0$, \begin{equation}\bar k(H, p, t)\leq  B(n, H, r_0)^2\bar k(H, p, r_0). \label{cc-com}\end{equation}
Now, by \eqref{dvol},  \eqref{lap-com1} and  \eqref{cc-com}, for each $t\geq r_0$
\begin{eqnarray}
\frac{d}{dt}\left(\ln\frac{\volume(B_t(x))}{\svolball{H}{t}}\right)&\leq& \left(-\kern-1em\int_{B_t(x)} \psi^{2p}\right)^{\frac{1}{2p}} \nonumber\\
&\leq & c(p, n)\, B(n, H, r_0) \bar k(H, p, r_0)^{\frac12},
\end{eqnarray}

Integral it from $r$ to $R$, we derive that
\begin{equation}\frac{\volume(B_R(x))}{\svolball{H}{R}}\leq e^{(R-r)c(n, p)B(n, H, r_0)\bar k(H, p, r_0)^{\frac{1}{2}}}\frac{\volume(B_r(x))}{\svolball{H}{r}}.
\label{grvol-com}
\end{equation}

In \eqref{eq-relative}, let $R=r_0$ and  $r\to 0$, we get that $\frac{\volume(B_{r_0}(x))}{\svolball{H}{r_0}}\leq \left(1-C_1(n, p, H, r_0)\bar k(H, p, r_0)^{\frac12}\right)^{-2p}\leq 1+ 8pC_1(n, p, H, r_0)\bar k(H, p, r_0)^{\frac12}$. And then by \eqref{grvol-com}, we have 
\begin{equation}\volume(B_R(x))\leq \left(1+8pC_1(n, p, H, r_0)\bar k(H, p, r_0)^{\frac12}\right)e^{(R-r_0)c(n, p)B(n, H, r_0)\bar k(H, p, r_0)^{\frac12}}\svolball{H}{R}.\label{gvol-com}\end{equation}
\end{proof}

Using \eqref{rvol-com}  direct calculation gives the following relative annulus:
\begin{Cor}\label{sph-comp}
	Given $n, p>\frac{n}{2}, H\leq 0$, let  $\delta_0=\delta(n, p, H, 1)$ be as in Lemma~\ref{doubling}. Then for any complete Riemannian $n$-manifold $M$ satisfies $\bar k(H, p, 1)\leq \delta\leq \delta_0$, any $R>r\geq1$, we have that
	\begin{equation}\frac{\volume(A_{r, R}(x))}{\volume(B_r(x))}\leq \frac{\svolann{H}{r, R}}{\svolball{H}{r}}+\left(e^{(R-r)c(n, p, H)\delta^{\frac12}}-1\right) \frac{\svolball{H}{R}}{\svolball{H}{r}}.  \label{ann-com} \end{equation}
\end{Cor}

\begin{proof}[Proof of Corollary~\ref{sph-comp}] 

For simplicity, we write $c$ instead of $c(n, p, H)$ in \eqref{rvol-com}. By \eqref{rvol-com},
\begin{eqnarray}
\frac{\volume(A_{r, R}(x))}{\svolann{H}{r, R}} & = & \frac{\volume(B_R(x))-\volume(B_r(x))}{\svolann{H}{r, R}}\nonumber\\
&\leq & \frac{e^{(R-r)c\delta^{\frac12}}\frac{\volume(B_r(x))}{\svolball{H}{r}}\svolball{H}{R}-\volume(B_r(x))}{\svolann{H}{r, R}}\nonumber\\
& = & \frac{\volume(B_r(x))}{\svolball{H}{r}}\, \frac{e^{(R-r)c\delta^{\frac12}}\svolball{H}{R}-\svolball{H}{r}}{\svolann{H}{r, R}}\nonumber\\
& = & \frac{\volume(B_r(x))}{\svolball{H}{r}}\left(1+\left(e^{(R-r)c\delta^{\frac12}}-1\right)\frac{\svolball{H}{R}}{\svolann{H}{r, R}}\right). \label{ann-com}
\end{eqnarray}
\end{proof}

\section{Volume entropy estimate and its application}

\subsection{Volume entropy estimate for integral Ricci curvature}
For a compact Riemannian $n$-manifold $M$, $h(M)$ measures the exponential growth rate of the volumes of balls in its universal cover $\tilde M$.  To estimate $h(M)$ where $M$ has integral Ricci curvature bound, the first thing one should consider is how the curvature condition of $\tilde M$ does.  
While any point wise curvature bounds lift to covering spaces directly, integral curvature bound is not as clear. Aubry \cite{Au2} derived the following relation between integral of function on the base manifold and the integral of its lift to the covering spaces. 

\begin{Thm}\cite[Lemma 1.0.1]{Au2}
Given $n, p>\frac{n}{2}, H\leq 0, d>0$, there is $\delta=\delta(n, p, d, H)>0,\ c(n, H, d)>0$ such that if a compact Riemannian $n$-manifold $M$ satisfies that
$$\bar k(H, p)\leq \delta, \quad \op{diam}(M)\leq d,$$
then for any non-negative function $f$ on $M$, for any normal cover $\bar{\pi}: \bar M\to M$, for any $\bar x\in \bar M$ and $R\geq 3d$, we have that 
$$\frac{1}{3^{n+1}e^{2(n-1)\sqrt{|H|}d}}-\kern-1em\int_M f\leq -\kern-1em\int_{B_R(\bar x)} f\circ \bar{\pi}\leq 3^{n+1}e^{2(n-1)\sqrt{|H|}d}-\kern-1em\int_M f.$$

In particular, if $\tilde M$ is the universal cover of $M$, then for any $\tilde x\in \tilde M$, any $R\geq 3d$, 
\begin{equation}
\left(-\kern-1em\int_{B_R(\tilde x)}\rho_H^p\right)^{\frac1p}\leq c(n, H, d)\, \bar k(H, p).  \label{univ-rho}
\end{equation}
\end{Thm}

By using \eqref{univ-rho}, the result of Theorem~\ref{main} is a direct application of \eqref{gvol-com} by taking $r_0=3\op{diam}(M)$
 or by \eqref{sph-com}. We give a short proof using \eqref{sph-com}.
 \begin{proof}[Proof of Theorem~\ref{main}] By definition \[
 	h(M) = \lim_{ R \ra \infty} \frac{\ln \vol (B_R(\tilde{x}))}{R}.  \] 
 	We can assume $\lim_{ R \ra \infty} \vol (B_R(\tilde{x})) = \infty$, otherwise $h(M) =0$ and there is nothing to prove. Then 
 	 \[
 	h(M)  =\lim_{R\to \infty}\frac{\volume(\partial B_R(\tilde x))}{\volume(B_R(\tilde x))}.  \]
 	Using \eqref{sph-com}, $$h(M) \le \lim_{R\to \infty} \left( \frac{\svolsp{H}{R}}{\svolball{H}{R}} + c(n, p) \, \bar k(H, p, R)^{\frac12}\right).$$ Since $\lim_{R\to \infty} \frac{\svolsp{H}{R}}{\svolball{H}{R}} = (n-1)\sqrt{-H}$,  and $\bar k(H, p, R) \le c(n, H, d)\, \bar k(H, p)$ for all $R \ge 3d$ by (\ref{univ-rho}), \eqref{h-est-int} follows.
\end{proof}
\subsection{Proof of Theorem~\ref{min-vol}}
 
To prove Theorem~\ref{min-vol}, we will argue by contradiction. Recall that for manifolds with integral curvature bounds, the following
 precompactness results holds.
\begin{Thm}\cite[Corollary 1.3]{PW1} \cite[Theorem 6.0.15]{Au2} \label{compact}
For $n\geq 2, p>\frac{n}{2}, H$, there exists $c(n, p, H)$ such that if a sequence of  compact Riemannian $n$-manifold $M_i$ satisfies that 
$\op{diam}(M_i)^2\bar{k_{M_i}}(H, p)\leq c(p, n,  H)$, then for any normal Riemannian cover $\pi_i: \bar M_i\to M_i$ and any $\bar x_i\in \bar M_i$, the sequence $(\bar M_i, \bar x_i)$ admits a subsequence that converges in the pointed Gromov-Hausdorff topology.
\end{Thm}

 For two compact $n$-manifolds $M_1, M_2$, we say $M_1$ dominates $M_2$ if there is a degree one continuous map $f: M_1\to M_2$. By Theorem~\ref{compact}, we may consider a sequence of Riemannian $n$-manifolds $M_i$ that dominates $Y_i$, a sequence of compact hyperbolic manifolds, and satisfies
\begin{equation}
\bar k_{M_i}(-1, p)\leq \delta_i\to 0,  \quad \frac{\volume(M_i)}{\volume(Y_i)}\leq 1+\epsilon_i\to 1, \quad \op{diam}(M_i)\leq d\label{min-ent}
\end{equation}
and $M_i\overset{GH}\to X$. 
Then by Corollary 1.4 in \cite{PW2},  showing that $X$ is isometric to a compact hyperbolic manifold will give a contradiction we need.

First recall  Besson, Courtois and Gallot's minimal volume entropy estimate \cite{BCG}:
\begin{Thm}\cite[Page 734]{BCG}  
	Let $Y$ be a compact hyperbolic $n$-manifold. For a compact Riemannian $n$-manifold $M$, if there is a continuous map $f: M \to Y$ with degree $k$, then 
	$$h^n(M)\volume(M)\geq k (n-1)^n\volume(Y).$$
	\label{min-vol-entr}
\end{Thm}
Hence combining \eqref{min-ent} we have \[
h(M_i) \ge (n-1) (1+ \epsilon_i)^{-\frac 1n}. \] On the other hand, 
by Theorem~\ref{main}, we have that
\begin{equation}h(M_i)\leq n-1+ c(n, p, d)\delta_i^{\frac12}. \label{vol-ent-small}\end{equation}
Therefore \begin{equation}
\lim_{i \to \infty} h(M_i) = n-1.  \label{eq-ent}
\end{equation}
 Theorem~\ref{min-vol-entr} and \eqref{vol-ent-small} also give that for $M_i$ in \eqref{min-ent}, $i$ large,
\begin{equation}\volume(M_i)\geq \left(\frac{n-1}{n-1+c\delta_i^{\frac12}}\right)^n\volume(Y_i)\geq v'(n)>0,\label{non-coll}\end{equation}
 where the last inequality follows from the Heintze-Margulis theorem (see e.g. \cite{BGS85}).

Now we want to show that \eqref{non-coll} implies that $M_i$ are uniformly locally noncollapsing  for all $i$ large, i.e.  there is $c(n,d)>0$, such that  for each $r<1$, $x_i\in M_i$, $\volume(B_r(x_i))\geq c(n,d)r^n$. This can be seen by \eqref{vol-double}. In fact, for $i$ large, and $0\leq r_1<r_2\leq d$, $x_i\in M_i$, 
\begin{equation}(1-\Psi(\delta_i | n, p, d))\frac{\svolball{-1}{r_1}}{\svolball{-1}{r_2}}\leq \frac{\volume(B_{r_1}(x_i))}{\volume(B_{r_2}(x_i))},\label{relative}\end{equation}
which implies 
\begin{equation}\volume(B_{r_1}(x_i))\geq (1-\Psi(\delta_i | n, p, d))\frac{\svolball{-1}{r_1}}{\svolball{-1}{d}}v'(n)>0.\label{loc-non-coll}\end{equation}

Recall that in \cite{BCG}, for any $0<c_i\to 0$, there are $C^1$-maps: $F_{c_i}: M_i\to Y_i$, with 
$$\left|\op{Jac}(F_{c_i})\right|\leq \left(\frac{h(M_i)+c_i}{n-1}\right)^n.$$
And then by the same argument as \cite[Lemma 3.13]{BBCG} by using $\volume(M_i)\leq (1+\epsilon_i)\volume(Y_i)$ and \eqref{loc-non-coll}, one can show that

\begin{Lem}[Almost 1-Lipschitz]
Given $n, p>\frac{n}2, d>0$, for $M_i$ as in \eqref{min-ent} and $F_{c_i}$ as above, we have that for $i$ large and $x_{i1}, x_{i2}\in M_i$, 
$$d(F_{c_i}(x_{i1}), F_{c_i}(x_{i2}))\leq (1+\Psi(\delta_i, \epsilon_i, c_i | n, p, d))d(x_{i1}, x_{i2})+\Psi(\delta_i, \epsilon_i, c_i | n, p, d).$$
\label{1-lip}
\end{Lem}
In the proof of Lemma~\ref{1-lip}, see  \cite[Lemma 3.13]{BBCG}, they used the segment inequality \cite[Theorem 2.11]{CC1}. Here we use the segment inequality \cite[Proposition 2.29]{TZ} and the fact that $|d_{x_i} F_{c_i}|_{C^0}\leq c(n)$ (\cite[Lemma 3.12]{BBCG}).

A corollary of Lemma~\ref{1-lip} is that, $\op{diam}(Y_i)\leq 2d$ for $i$ large. In fact, for each $y_1, y_2\in Y_i$, as $F_{c_i}$ is one degree, there is  $x_1, x_2\in X_i$ such that $F_{c_i}(x_j)=y_j$, $j=1, 2$. Then
$$d(y_1, y_2)\leq (1+\Psi(\delta_i, \epsilon_i, c_i | n, p, d))d(x_{1}, x_{2})+\Psi(\delta_i, \epsilon_i, c_i | n, p, d)\leq 2d.$$

Now by Cheeger's finiteness theorem, without loss of generality, we can assume that $Y_i=Y$, for all $i$, in \eqref{min-ent}.

Let $\psi_i: X\to M_i$ be a $\sigma_i$-Gromov-Hausdorff approximatons, $\sigma_i\to 0$ and let $F_i=F_{c_i}\circ \psi_i : X\to Y$. By passing to a subsequence, $F_i\to F : X\to Y$.   And then Theorem~\ref{min-vol} is derived if $F$ is a isometry:
\begin{Thm}
The map $F : X\to Y$ above is a isometry.
\label{iso}
\end{Thm}

To prove Theorem~\ref{iso}, we need the following results for the Gromov-Hausdorff limit space $X$ of a sequence of complete Riemannian $n$-manifolds $M_i$ with
\begin{equation}\bar k_{M_i}(H,p, 1)\leq \delta_i\to 0, \quad \volume(B_1(q_i))\geq v>0, \forall q_i\in M_i.\label{non-col}\end{equation}
With \cite[Theorem 2.33]{TZ} and the splitting theorem \cite[page 473]{PW2}, volume comparison (Lemma~\ref{doubling}),  volume convergence \cite[Theorem 1.3]{PW2}, the proof is the same as in \cite{CC2, CC3}. 
\begin{Thm}
{\rm(3.6.1)} The regular set of $X$ is dense, where $x\in X$ is called regular if its tangent cones are isometric to $\Bbb R^n$.

{\rm(3.6.2)} Let $\mathcal{R}_{\tau}$ be the set of points $x\in X$ such that there is $r_0>0$, for $s\leq r_0$, $d_{GH}(B_s(x), \sball{0}{s})\leq \tau s$. Then the interior set of $\mathcal{R}_{\tau}$, $\mathring{\mathcal{R}}_{\tau}$, is open and connected and has a natural smooth manifold structure whose metric is bi-h$\ddot{o}$lder to a smooth Riemannian metric.

{\rm (3.6.3)} The $n-2$ Hausdorff measure $\op{Haus}^{n-2}(\mathcal{S})=0$.

{\rm(3.6.4)} For $x_i\in M_i, x_i\to x\in X$, $r>0$, $\lim_{i\to\infty}\volume(B_r(x_i))=\op{Haus}^n(\volume(B_r(x)))$.
\label{reg-lim}
\end{Thm}

 Then a similarly argument as in \cite{BBCG} by using \eqref{relative} and Theorem~\ref{reg-lim} gives Theorem~\ref{iso}.


\end{document}